\documentclass[12pt]{amsart}
\usepackage{amssymb,amscd}
\usepackage{verbatim}

\usepackage{xcolor}

\usepackage{amsmath,amssymb,graphicx,mathrsfs}   
\usepackage{enumerate}
\usepackage[colorlinks=true,allcolors = blue]{hyperref} 

\usepackage{tikz}
\usetikzlibrary{matrix}

\usepackage[all]{xy}

\usepackage{amssymb,amsfonts,amsthm,amsmath,calligra}
\usepackage{slashed}
\usepackage{yfonts}
\usepackage{mathrsfs,pifont}
\usepackage{float}

\usepackage[all]{xy}

\usepackage{tikz}

\usepackage{graphicx}
\usepackage{xcolor}

\usepackage{amssymb,amsfonts,amsthm,amsmath}
\usepackage[all]{xy}
\usepackage{slashed}
\usepackage{yfonts}
\usepackage{mathrsfs,pifont}

\let\frak\mathfrak

\def\>{\relax\ifmmode\mskip.666667\thinmuskip\relax\else\kern.111111em\fi}
\def\<{\relax\ifmmode\mskip-.333333\thinmuskip\relax\else\kern-.0555556em\fi}
\def\vsk#1>{\vskip#1\baselineskip}
\def\vv#1>{\vadjust{\vsk#1>}\ignorespaces}
\def\vvn#1>{\vadjust{\nobreak\vsk#1>\nobreak}\ignorespaces}

  \let\ssize\scriptstyle
\let\sssize\scriptscriptstyle

\let\Medskip\medskip
\def\medskip{\par\Medskip}
\let\Bigskip\bigskip
\def\bigskip{\par\Bigskip}

\let\Maketitle\maketitle
\def\maketitle{\Maketitle\thispagestyle{empty}\let\maketitle\empty}

\newtheorem{thm}{Theorem}[section]
\newtheorem{cor}[thm]{Corollary}

\theoremstyle{definition}                                  
\numberwithin{equation}{section}

\theoremstyle{definition}

\let\mc\mathcal
\let\nc\newcommand

\let\la\lambda

\let\phi\varphi

\let\der\partial

\let\geq\geqslant

\let\leq\leqslant

\let\on\operatorname
\let\bi\bibitem
\let\bs\boldsymbol

\def\C{{\mathbb C}}
\def\Z{{\mathbb Z}}

\def\F{{\mathbb F}}

\def\+#1{^{\{#1\}}}

\def\beq{\begin{equation}}
\def\eeq{\end{equation}}
\def\be{\begin{equation*}}
\def\ee{\end{equation*}}

\nc{\bea}{\begin{eqnarray*}}
\nc{\eea}{\end{eqnarray*}}
\nc{\bean}{\begin{eqnarray}}
\nc{\eean}{\end{eqnarray}}

\nc{\Il}{{\mc I_{\bs\la}}}
\nc{\bla}{{\bs\la}}
\nc{\Fla}{\F_\bla}
\nc{\tfl}{{T^*\Fla}}
\nc{\GL}{{GL_n(\C)}}
\nc{\GLC}{{GL_n(\C)\times\C^*}}

\let\sd s 

\def\ddk_#1{\kk_{#1}\<\>\frac\der{\der\<\>\kk_{#1}}}

\def\bul{\mathbin{\raise.2ex\hbox{$\sssize\bullet$}}}
\def\intt{\mathchoice
{\mathop{\raise.2ex\rlap{$\,\,\ssize\backslash$}{\intop}}\nolimits}
{\mathop{\raise.3ex\rlap{$\,\sssize\backslash$}{\intop}}\nolimits}
{\mathop{\raise.1ex\rlap{$\sssize\>\backslash$}{\intop}}\nolimits}
{\mathop{\rlap{$\sssize\<\>\backslash$}{\intop}}\nolimits}}

\let\kk q 
\let\cc c

\let\Ko K

\def\GZ/{Gelfand-Zetlin}
\def\KZ/{{\slshape KZ\/}}
\def\qKZ/{{\slshape qKZ\/}}
\def\XXX/{{\slshape XXX\/}}

\nc{\A}{{\mc A}}

\nc{\hsl}{\widehat{{\frak{sl}_2}}}

\nc{\BC}{{ \mathbb C}}
\nc{\lra}{\longrightarrow}
\nc{\CO}{{\mathcal{O}}}
\nc{\BZ}{{ \mathbb Z}}
\nc{\hfn}{\hat{\frak{n}}}
\nc\Zs{{\Z/p^s\Z}}
\nc\Zo{{\Zs[z]^0}}
\nc\gr{{\on{gr}}}

\nc\fD{{\frak D}}

\newcommand{\T}{\mathsf{T}}

\newcommand{\matZ}{\mathbb{Z}}

\newcommand{\ovl}{\overline}

\newcommand{\bv}{\mathsf{v}}

\usepackage{tikz}

\usetikzlibrary{decorations}
\usetikzlibrary{decorations.pathmorphing}
\usetikzlibrary{calc}

\newsavebox{\Ipm}
\savebox{\Ipm}{%
\begin{tikzpicture}[baseline= {($(current bounding box.base)-(0pt,-20pt)$)}]
\begin{scope}
\draw[dashed,line width=0.35mm] (0,0)-- (10,0) ;
\node[circle,draw,minimum size=4mm,fill=black] (c) at (0,0){};
\node[circle,draw,minimum size=4mm,fill=black] (c) at (2,0){};
\node[circle,draw,minimum size=4mm,fill=black] (c) at (4,0){};
\node[circle,draw,minimum size=4mm,fill=black] (c) at (6,0){};
\node[circle,draw,minimum size=4mm,fill=black] (c) at (8,0){};
\node[circle,draw,minimum size=4mm,fill=black] (c) at (10,0){};
\draw [-to,line width=0.5mm](0,0) -- (1.8,0);
\draw [-to,line width=0.5mm](8,0) -- (9.8,0);
\node[rectangle,draw,minimum size=4mm,fill=black] (c) at (4,-2){};
\draw [-to,line width=0.5mm](4,-2) -- (4,-0.2);
\node[rectangle,draw,minimum size=4mm,fill=black] (c) at (6,-2){};
\draw [-to,line width=0.5mm](6,-2) -- (6,-0.2);
\node at (0,0.5) { $\bv_1$};
\node at (2,0.5) { $\bv_2$};
\node at (4,0.5) { $\bv_k$};
\node at (6,0.5) { $\bv_{n-k}$};
\node at (8,0.5) { $\bv_{n-2}$};
\node at (10,0.5) { $\bv_{n-1}$};
\node at (4,-2.6) { $z_1$};
\node at (6,-2.6) { $z_2$};
\end{scope}
\end{tikzpicture}}

\begin{document}

\hrule width0pt
\vsk->

\title[Dwork congruences via $q$-deformation]
{Dwork congruences via $q$-deformation}

\author
[Pavan Kartik and Andrey Smirnov]
{Pavan Kartik$^{\diamond}$ and Andrey Smirnov$^{\star}$ }

\maketitle

\begin{center}
{ Department of Mathematics, University
of North Carolina at Chapel Hill\\ Chapel Hill, NC 27599-3250, USA\/}

\end{center}

\vsk>
{\leftskip3pc \rightskip\leftskip \parindent0pt \Small
{\it Key words\/}:  Dwork congruences, vertex functions, $q$-deformation.

\vsk.6>
{\it 2020 Mathematics Subject Classification\/}: 
\par}


{\let\thefootnote\relax
\footnotetext{\vsk-.8>\noindent
$^\star\<${\sl E\>-mail}:\enspace asmirnov@email.unc.edu
\\
$^\diamond\<${\sl E\>-mail}:\enspace  pkartik1@unc.edu}}

\begin{abstract}
We consider a system of polynomials $T_{s}(z,q) \in \mathbb{Z}[z,q]$ which appear as truncations of the K-theoretic vertex function for the cotangent bundles over Grassmannians $T^{*}Gr(k,n)$. We prove that these polynomials satisfy a natural $q$-deformation of Dwork's congruences
$$
\dfrac{T_{s+1}(z,q)}{T_{s}(z^p,q^p)} \equiv  \dfrac{T_{s}(z,q)}{T_{s-1}(z^p,q^p)}  \pmod{[p^s]_q}  
$$
In the limit $q \to 1$ we recover the main result of \cite{SV}.

\end{abstract}


\setcounter{footnote}{0}
\renewcommand{\thefootnote}{\arabic{footnote}}

\section{Introduction}

\subsection{}

In his study of zeta functions of families of algebraic varieties \cite{Dw69} Dwork encountered remarkable congruences. These congruences are satisfied by truncations of the period integrals, which describe the Gauss-Manin connection for such families. The basic example here is the Legendrian family of elliptic curves
\bean \label{ellcurve}
y^2= x (1-x) (1-z x) 
\eean
parametrized by $z\in \mathbb{C}\setminus\{0,1\}$. Let $\omega = dx/y$ be the holomorphic differential on the curve and let
\bean \label{Eulerint}
F(z)=  \dfrac{1}{\pi} \oint_{\gamma} \omega = \dfrac{1}{\pi}\int\limits_{0}^{1} \, \dfrac{dx}{\sqrt{x (1-x) (1-z x)}} = \sum\limits_{n=0}^{\infty} \binom{-1/2}{n}^2 z^n\, 
\eean 
be its period over the $a$-cycle $\gamma$. As a function of the parameter $z$ this period is described by the classical hypergeometric series
\bean \label{hyperghal}
F(z)=  {}_{2}F_{1}(1/2,1/2;1;z)
\eean 
which satisfies the well-known hypergeometric differential equation.

Let $p$ be a prime number. Dwork found that the truncated power series
\bean \label{dwocong}
T_{s}(z) = \sum\limits_{n=0}^{p^s-1} \binom{-1/2}{n}^2 z^n, \ \ \ s =1,2,\dots
\eean 
satisfy the following congruences
$$
\dfrac{T_{s+1}(z)}{T_{s}(z^p)} \equiv  \dfrac{T_{s}(z)}{T_{s-1}(z^p)}\, \ \ \pmod{p^s} 
$$
Let $z \in \mathbb{Z}_p$ be a $p$-adic integer with residue $z \pmod{p} =z_0 \in \mathbb{F}_p$. Among other things, the congruences (\ref{dwocong}) imply that over $\mathbb{Q}_p$ the limit 
\bean \label{unroot}
\lambda(z_0)=(-1)^{\frac{p-1}{2}}\lim\limits_{s\to \infty} \dfrac{T_{s+1}(z)}{T_{s}(z^p)} 
\eean
exists and is equal to the unit root of zeta function of the curve (\ref{ellcurve}) with parameter $z_0$. The second root is equal to $p/\lambda(z_0)$. In this way, the congruences (\ref{dwocong}) fully determine the zeta function of (\ref{ellcurve}), i.e., the numbers of points on (\ref{ellcurve}) over finite fields $\mathbb{F}_{p^l}$. 
 
\subsection{}
Since then Dwork's results have been expanded considerably and general theories explaining the underlying mechanism of such congruences have been developed, see \cite{K85,BV1} for examples. 

It was observed in \cite{SV} that interesting examples of Dwork congruences appear naturally for the {\it cohomological vertex functions} of Nakajima varieties. These functions control the equivariant quantum cohomology of quiver varieties \cite{MO19} and 
can be viewed as far reaching generalizations of the classical hypergeometric series.  In fact, the  classical hypergeometric functions are the vertex functions for the simplest Nakajima varieties given by the cotangent bundles over projective spaces 
$T^{*} \mathbb{P}^n$.

In \cite{SV} an analog of truncations (\ref{dwocong}) for the vertex functions of  the cotangent bundles over Grassmannians $T^{*} Gr(k,n)$ were considered. It was proven that these polynomials also satisfy the Dwork's congruences. 

It is known that cohomological vertex functions are limits of a more general and richer structures  - the {\it K-theoretic vertex functions} \cite{Oko17}. For simple cases, e.g., $T^{*} \mathbb{P}^n$ these functions are given by the standard $q$-hypergeometric series. The cohomological vertex functions can be viewed as specializations at $q=1$ of the K-theoretic vertex functions. 

It is therefore natural to ask whether a natural $q$-deformation of the Dwork's congruences exists. In this paper we give a positive answer to this question:  we consider the same Nakajima variety $T^{*} Gr(k,n)$ as in \cite{SV} and for a prime $p$ we define polynomials 
$$
T_{s}(z,q) \in  \mathbb{Z}[z,q], \ \ s \in \mathbb{N},
$$
which are certain truncations of the corresponding K-theoretic vertex function. Specialized  at $q=1$ these polynomials coincide with the truncations considered in \cite{SV}. As our main result, we prove the following 
\begin{thm} \label{mainthm}
The polynomials $T_{s}(z,q)$ satisfy the $q$-deformed Dwork's congruences
\bean \label{congint}
\dfrac{T_{s+1}(z,q)}{T_{s}(z^p,q^p)} \equiv  \dfrac{T_{s}(z,q)}{T_{s-1}(z^p,q^p)}  \pmod{[p^s]_q}  
\eean 
\end{thm}
Let us explain the notations. For a natural number $n$ we denote by
$$
[n]_q = \dfrac{1-q^n}{1-q} = 1+q+\dots+q^{n-1}
$$
the corresponding $q$-number. In particular, for a prime $p$, $[p]_q= 1+ q+ \dots + q^{p-1} = \varphi_p(q)$ is the $p$-th cyclotomic polynomial. For a power $p^s$ we have
\bean \label{pqnum}
[p^s]_q = \varphi_{p}(q) \varphi_{p}(q^p) \dots  \varphi_{p}(q^{p^{s-1}})
\eean
The congruence (\ref{congint}) means that the difference of the left and right sides is divisible by the polynomial $[p^s]_q$.

In the limit $q\to 1$, the polynomials $T(z,q)$ specialize to the polynomials $T_s(z)$ considered in \cite{SV}. The $q$-number $[p^s]_q$ specializes to $p^s$ in this limit and as a simple corollary of the above Theorem we arrive at the main result of \cite{SV}. 

\subsection{}
We note that compared to \cite{SV} our proof of $q$-congruences is completely elementary. To prove Theorem \ref{mainthm} it is enough to check that the congruence (\ref{congint}) becomes identity
$$
\dfrac{T_{s+1}(z,q)}{T_{s}(z^p,q^p)} =  \dfrac{T_{s}(z,q)}{T_{s-1}(z^p,q^p)}
$$
when $q$ is specialized to a non-trivial root of equation $X^{p^s}=1$ (note that $[p^s]_q=0$ when $q$ is such a root). The combinatorial proof we describe here is essentially by iterative use of an elementary identity: 
$$
(1-x)(1-x q) \dots (1-x q^{l-1})  = (1-x^l)
$$
when $q$ is a primitive $l$-th unit root.

\subsection{} 
As in (\ref{unroot}), the $q$-congruences  (\ref{congint})  imply that the limit 
$$
\lambda(z,q)=(-1)^{\frac{p-1}{2}}\lim\limits_{s\to \infty} \dfrac{T_{s+1}(z,q)}{T_{s}(z^p,q^p)} 
$$
considered over $\mathbb{Q}_p$, exists. In fact it follows from our proof that when $q$ is specialized to a unit root
$$
\lambda(z,q) = (-1)^{\frac{p-1}{2}}\lim\limits_{s\to \infty}\, \dfrac{T_{s+1}(z,q)}{T_{s}(z^p,q^p)} =\dfrac{T_{L+1}(z,q)}{T_{L}(z^p,q^p)}
$$
for sufficiently large $L$. In particular, for $q$ given by $p$-adic unit roots, $\lambda(z,q)$ is a rational function of $z$. 

The element $\lambda(z,q)$ is a certain $q$-deformation of the unit root for the corresponding zeta function, as we discussed above. We do not know if the $q$-deformation $\lambda(z,q)$ has 
a similar arithmetic significance, or if it is related to natural $q$-deformations arising in arithmetic geometry  \cite{SCH}. One may expect that $\lambda(z,q)$ is one of the ``eigenvalues'' of the $q$-deformed Frobenius intertwiner considered for the  
$q$-hypergeometric case in  \cite{S}, see also \cite{KS}.

\section*{Acknowledgments} This work is
supported by NSF grants DMS - 2054527, DMS-2401380 and by the Simons Foundation grant “Travel Support for Mathematicians”. We thank A.\,Varchenko for useful discussions. 

\section{Polynomials $T_s(z,q)$} 
In this section we define polynomials $T_s(z,q)$ which are truncations of the $K$-theoretic vertex functions for Nakajima variety $X=T^*Gr(k,n)$. This definition is inspired by $3D$-mirror symmetry which equips  the vertex functions with integral representations, which are analogs of Euler's integral representation (\ref{Eulerint}) for the hypergeometric function (\ref{hyperghal}). We refer to \cite{SV} for details and motivations.

\subsection{Superpotential function}
For the rest of this note we assume that 

\[\omega=\frac{r}{q}\leq\frac{1}{2},\ r,q \text{ positive integers} \]
\\ 

Let $p$ be an odd prime of the form $p=lq+1,\ \ l$ a positive integer. Let $n,\ k$ be positive integers so that $n\geq 2k$. We associate an $A_{n-1}$ framed quiver to each pair $(k,n)$ with vertices labelled $1,2,\dots n-1$ as follows:\\
The quiver has nontrivial one-dimensional framings at vertices $k$ and $n-k$, and we define the dimension vector by
 \[\text{v}_{i}:=   \begin{cases}
        i & i<k\\
        k & k\leq i\leq n-k\\
        n-i & n-k < i
    \end{cases}
    \]
To a vertex $i$ with dimension $\text{v}_{i}$ we associate to it variables $x_{i,j}$ where $1\leq j\leq \text{v}_{i}$, and to a framing vertex we associate variables $z_{i,j}$ where $1\leq j\leq \text{w}_{i}$ (where $\text{w}_{i}$ is 1, since we consider 1-dimensional framings at vertices $k$ and $n-k$), see figure below. 
\begin{figure}[h!]
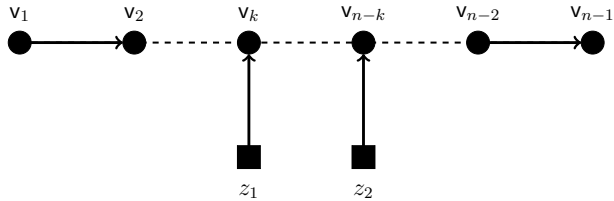

$$
\resizebox{0.5\textwidth}{!}{%
\usebox{\Ipm}}
$$
\caption{$A_{n-1}$ quiver with two framings} 
\end{figure}

\begin{itemize}
    \item To an arrow from vertex $j$ to vertex $i$ we associate a factor 
    \[\prod_{a=1}^{\text{v}_{i}}\prod_{b=1}^{\text{v}_{j}}(x_{i,a}-x_{j,b})^{-\omega}\]
    \item To a vertex $m$ of the quiver we associate a factor 
    \[\prod_{1\leq i<j\leq \text{v}_{m}}(x_{m,i}-x_{m,j})^{2\omega}\]
    \item To a vertex $m$ of the quiver we associate a factor 
    \[\prod_{j=1}^{\text{v}_{m}}(x_{m,j})^{-1+\omega}\]
\end{itemize}

These rules give us the following superpotential, denoted by $\Phi$
\begin{multline}
\Phi(x,z):=\left(\prod_{i=1}^{n-1}\prod_{j=1}^{\text{v}_{i}}x_{i,j}\right)^{-1+\omega}\times\prod_{m=1}^{n-1}\prod_{1\leq i<j\leq \text{v}_{m}}(x_{m,i}-x_{m,j})^{2\omega}\\
\times \prod_{i=1}^{n-2}\prod_{a=1}^{\text{v}_{i}}\prod_{b=1}^{\text{v}_{i+1}}(x_{i,a}-x_{i+1,b})^{-\omega}\times\prod_{i=1}^{k}((z_{k,1}-x_{k,i})(z_{n-k,1}-x_{n-k,i}))^{-\omega}
\end{multline}
We consider the specialization $z_{k,1}=z,\ z_{n-k,1}=1$, and if $n-k=k$ we assume $z_{k,1}=z$ and $z_{k,2}=1$.
We also define 
\[\Delta(x):=\prod_{m=1}^{n-1}\prod_{1\leq i<j\leq \text{v}_{m}}(x_{m,j}-x_{m,i})\]
and 
\[\Phi_{s}(x,z):=(\Delta(x))^{p^{s}}\times(\Phi(x,z))^{1-p^{s}}\]
\[\Phi_{s}(x,z)=\overline{\Phi}_{s}(x,z)\Delta(x)\]
\\
Let us define polynomials
\bean \label{tpols} T_{s}(z)=(-1)^{(p^{s}-1)N\omega}\text{coeff}_{x^{dp^{s}-1}}\Phi_{s}(x,z)
\eean
where $\text{coeff}_{x^{dp^{s}-1}}$ denotes the coefficients of the monomial 
\[x^{dp^{s}-1}:=\prod_{i=1}^{n-1}\prod_{j=1}^{\text{v}_{i}}x_{i,j}^{jp^{s}-1}\]
in variables $x_{i,j}$. The integer $N$ is defined in Theorem 3.2 of \cite{SV} and simply fixes the sign so that $\T_s(0)=1$.  The following is the main result of \cite{SV}. 
\begin{thm} 
\label{dworkth}
The polynomials $T_s(z)$ satisfy the Dwork-type congruences:
\bean
 \label{dworkid}
\dfrac{T_{s+1}(z)}{T_s(z^p)}=\dfrac{T_{s}(z)}{T_{s-1}(z^p)} \mod p^s
\eean
with $s=1,2,\dots$ and $T_{0}(z)=1$. 
\end{thm}

\subsection{$q$-deformation \label{signref}}
Consider the following $q$-deformation of $\ovl{\Phi}_{s}(x,z)$ given by 
\begin{align} \label{Fbar}
\overline{\Phi}_{s}(x,z,q)=(\prod_{m=1}^{n-1}\prod_{j=1}^{\text{v}_{m}}x_{m,i})^{(p^{s}-1)(1-\omega)}\times(\prod_{m=1}^{n-1}\prod_{1\leq i\neq j\leq \text{v}_{m}}\prod_{r=0}^{\frac{(p^{s}-1)(1-2\omega)}{2}-1}(x_{m,i}-q^{r}x_{m,j}))\\ \nonumber
\times(\prod_{i=1}^{n-2}\prod_{a=1}^{\text{v}_{i}}\prod_{b=1}^{\text{v}_{i+1}}\prod_{r=0}^{(p^{s}-1)\omega-1}(x_{i,a}-q^{r}x_{i+1,\ b}))\times(\prod_{i=1}^{k}\prod_{r=0}^{(p^{s}-1)\omega-1}(z-q^{r}x_{k,i})(1-q^{r}x_{n-k,\ i}))
\end{align}
And we define 
\begin{align} \label{F}
\Phi_{s}(x,z,q):=\Delta(x)\overline{\Phi}_{s}(x,z,q)
\end{align}
The group $S_{\text{v}}:=S_{\text{v}_{1}}\times\dots\times S_{\text{v}_{n-1}}$ acts on the set of variables $\{x_{i,j}\}_{1\leq i\leq n-1,\ 1\leq j\leq v_{i}}$ by permuting the coordinates as follows:
    \begin{align*}
        \sigma\cdot x_{ij}=x_{i\sigma(j)}
    \end{align*}
Note that in our definition $\ovl{\Phi}_s$ is symmetric under $S_{\text{v}}$, and $\Delta$ is skew symmetric, so $\Phi_{s}$ is skew symmetric. 
In the $q=1$ case, $\overline{\Phi}_{s}(x,z,q)$ specializes to $(-1)^{\theta_{s}}\overline{\Phi}_{s}(x,z)$, where \[\theta_{s}:=\left(\sum_{m=1}^{n-1}\frac{\text{v}_{m}(\text{v}_{m}-1)}{2}\right)\frac{(p^{s}-1)(1-2\omega)}{2}\]\\
We define polynomials $T_{s}(z,q) \in \matZ[z,q]$ by 
\bean \label{tqdef}
T_{s}(z,q)=(-1)^{(p^{s}-1)N\omega}\text{coeff}_{x^{dp^{s}-1}}\Phi_{s}(x,z,q)
\eean 
with the same notations as above. Note that at $q=1$, the polynomials $\Phi_{s}(x,z,q)$ and $T_{s}(z,q)$ specialize to $(-1)^{\theta_{s}}\Phi_{s}(x,z)$ and $(-1)^{\theta_{s}}T_{s}(z)$ respectively. \\

\noindent
{{\bf Example:} 
If $n=2$, $\omega=\frac{1}{2}$ and $k=1$ we have $n-1=1$ vertex with $\text{v}_{1}=1$, and the variable associated to this vertex is $x_{1,1}$, which we will denote by $x$. According to the rules mentioned above, we have 

\[\Phi_{s}(x,z,q)=x^{\frac{(p^{s}-1)}{2}}\times\prod_{r=0}^{\frac{(p^{s}-1)}{2}-1}(z-q^{r}x)(1-q^{r}x)\]
and thus 
\[T_{s}(z,q)=\text{coeff}_{x^{p^{s}-1}}\left((-1)^{\frac{(p^{s}-1)}{2}}x^{\frac{p^{s}-1}{2}}\prod_{r=0}^{\frac{p^{s}-1}{2}-1}(z-q^{r}x)(1-q^{r}x)\right).\]
\section{$q$-deformation of Dwork's congruences} 

\subsection{$q$-deformed Dwork's congruences}

Here is our main result. 

\begin{thm}
The polynomials $T_{s}(z,q)$ satisfy $q$-deformed Dwork's congruence
\bean \label{mainform}
\dfrac{T_{s+1}(z,q)}{T_{s}(z^p,q^p)} \equiv  \dfrac{T_{s}(z,q)}{T_{s-1}(z^p,q^p)}  \pmod{[p^s]_q}  
\eean 
where $[p^s]_q$ denotes the $q$-number (\ref{pqnum}).
\end{thm}

\begin{proof}
    It is enough to show that 
    \[\frac{T_{s+1}(z,q)}{T_{s}(z^{p},q^{p})}=\frac{T_{s}(z,q)}{T_{s-1}(z^{p},q^{p})}\]
    when $q$ is a root of $X^{p^{s}}=1,\ q\neq 1$.\newline
    We can find bounds on the degree of each variable $x_{ij}$ by considering the following cases:
    \begin{itemize}
    \item If $k=1$, then $\text{v}_{i}=1$ for each $ 1\leq i\leq n-1$. If $i\neq 1$ and $i\neq n-1$, then by a degree count of (\ref{F}), we conclude that
    \[\text{deg}_{x_{ij}}(\Phi_{s})=(p^{s}-1)(1-\omega)+2(p^{s}-1)\omega\]
    \item If $k=1$ and $i=1$ or $i=n-1$, then by a similar degree counting argument of (\ref{F}) we get
    \[\text{deg}_{x_{ij}}(\Phi_{s})=(p^{s}-1)(1-\omega)+(p^{s}-1)\omega-1+(p^{s}-1)\omega-1\]
    \item If $k\neq 1$ and $i=1$ or $i=n-1$, then $v_{i}=1$, so the degree of $x_{i,1}$ in $\Phi_{s}(x,z)$ (hence in $\Phi_{s}(x,z,q)$) is 
    \[\text{deg}_{x_{ij}}(\Phi_{s}(x,z,q))=(p^{s}-1)(1-\omega)+(p^{s}-1)\omega=\text{v}_{i}p^{s}-1\]
    \item If $k\neq 1$ and $i\neq 1$ and $i\neq n-1$, then 
    \[\text{deg}_{x_{ij}}(\Phi_{s}(x,z,q))=\text{v}_{i}p^{s}-1+\omega(p^{s}-1)\]
    \end{itemize}
    So we observe that in each of these cases we have 
    \begin{align} \label{degbound}
    \text{deg}_{x_{ij}}(\Phi_{s}(x,z,q))<(\text{v}_{i}+1)p^{s}-1\end{align}
    Recall that the group $S_{\text{v}}:=S_{\text{v}_{1}}\times\dots\times S_{\text{v}_{n-1}}$ acts on the set of variables $\{x_{i,j}\}_{1\leq i\leq n-1,\ 1\leq j\leq v_{i}}$ by permuting the coordinates as follows:
    \begin{align*}
        \sigma\cdot x_{ij}=x_{i\sigma(j)}
    \end{align*}
    Hence $S_{\text{v}}$ also acts on the polynomial $\Phi_{s}(x,z,q)$. By the skew-symmetry of $\Phi_{s}(x,z,q)$ observed in Section \ref{Fbar}, we conclude that 
    \[(\sigma\cdot\Phi_{s})(x,z,q)=\Phi_{s}(\sigma\cdot x,z,q)=\epsilon(\sigma)\Phi_{s}(x,z,q)\]
    If $l<s$, we have
    \begin{align*}
        (p^{s}-1)(1-\omega)-(p^{l}-1)(1-\omega)&=p^{l}(p^{s-l}-1)(1-\omega)\\
        \frac{(p^{s}-1)(1-2\omega)}{2}-\frac{(p^{l}-1)(1-2\omega)}{2}&=p^{l}\left(\frac{(p^{s-l}-1)(1-2\omega)}{2}\right)\\        
        (p^{s}-1)\omega-(p^{l}-1)\omega&=p^{l}(p^{s-l}-1)\omega
    \end{align*}
    Hence, if $q$ is a primitive $p^{l}-$th root of unity (with $0<l<s$), we have 
    \begin{multline*}
\prod_{m=1}^{n-1}\prod_{1\leq i\neq j\leq \text{v}_{m}}\prod_{r=0}^{\frac{(p^{s}-1)(1-2\omega)}{2}-1}(x_{m,i}-q^{r}x_{m,j}) = \\
 \prod_{m=1}^{n-1}\prod_{1\leq i\neq j\leq \text{v}_{m}}\prod_{r=0}^{\frac{(p^{l}-1)(1-2\omega)}{2}-1}(x_{m,i}-q^{r}x_{m,j})
\times\left(\prod_{m=1}^{n-1}\prod_{1\leq i\neq j\leq \text{v}_{m}}(x_{m,i}^{p^{l}}-x_{mj}^{p^{l}})\right)^{\frac{(p^{s-l}-1)(1-2\omega)} {2}}  
\end{multline*}
similarly
\begin{multline*}
 \prod_{i=1}^{n-2}\prod_{a=1}^{\text{v}_{i}}\prod_{b=1}^{\text{v}_{i+1}}\prod_{r=0}^{(p^{s}-1)\omega-1}(x_{i,a}-q^{r}x_{i+1,\ b})=\\\prod_{i=1}^{n-2}\prod_{a=1}^{\text{v}_{i}}\prod_{b=1}^{\text{v}_{i+1}}\prod_{r=0}^{(p^{l}-1)\omega-1}(x_{i,a}-q^{r}x_{i+1,\ b})\times(\prod_{i=1}^{n-2}\prod_{a=1}^{\text{v}_{i}}\prod_{b=1}^{\text{v}_{i+1}}(x_{i,a}^{p^{l}}-x_{i+1,\ b}^{p^{l}}))^{(p^{s-l}-1)\omega}   
\end{multline*}
also
\begin{multline*}
\prod_{i=1}^{k}\prod_{r=0}^{(p^{s}-1)\omega-1}(z-q^{r}x_{k,i})(1-q^{r}x_{n-k,\ i})=\\ \prod_{i=1}^{k}\prod_{r=0}^{(p^{l}-1)\omega-1}(z-q^{r}x_{k, i})(1-q^{r}x_{n-k,\ i})
 \times(\prod_{i=1}^{k}(z^{p^{l}}-x_{k,i}^{p^{l}})(1-x_{n-k,\ i}^{p^{l}}))^{(p^{s-l}-1)\omega}    
\end{multline*}
and finally
\begin{multline*}
 \prod_{m=1}^{n-1}\prod_{i=1}^{\text{v}_{m}}x_{m,i}^{(p^{s}-1)(1-\omega)}=\prod_{m=1}^{n-1}\prod_{i=1}^{\text{v}_{m}}x_{m,i}^{(p^{l}-1)(1-\omega)}\times(\prod_{m=1}^{n-1}\prod_{i=1}^{\text{v}_{m}}x_{m,i}^{p^{l}})^{(p^{s-l}-1)(1-\omega)}
\end{multline*}

Using the last four equalities, from (\ref{Fbar}) we find that:
\begin{align} \label{split1}
\overline{\Phi}_{s+1}(x,z,q)&=\overline{\Phi}_{l}(x,z,q)\overline{\Phi}_{s-l+1}(x^{p^{l}},z^{p^{l}},1)
\end{align}
\begin{align} \label{split2}
\overline{\Phi}_{s}(x,z^{p},q^{p})&=\overline{\Phi}_{l-1}(x,z^{p},q^{p})\overline{\Phi}_{s-l+1}(x^{p^{l-1}},z^{p^{l}},1)
\end{align}
\begin{align} \label{split3}
\overline{\Phi}_{s}(x,z,q)&=\overline{\Phi}_{l}(x,z,q)\overline{\Phi}_{s-l}(x^{p^{l}},z^{p^{l}},1)
\end{align}
\begin{align} \label{split4}
\overline{\Phi}_{s-1}(x,z^{p},q^{p})&=\overline{\Phi}_{l-1}(x,z^{p},q^{p})\overline{\Phi}_{s-l}(x^{p^{l-1}},z^{p^{l}},1) 
\end{align}
\newline We can rewrite (\ref{split2}) and (\ref{split3}) by replacing $x$ with $x^{p}$ to get:
\begin{align} \label{reqd1}
    \overline{\Phi}_{s}(x^{p},z^{p},q^{p})&=\overline{\Phi}_{l-1}(x^{p},z^{p},q^{p})\overline{\Phi}_{s-l+1}(x^{p^{l}},z^{p^{l}},1)
\end{align}
\begin{align} \label{reqd2}
    \overline{\Phi}_{s-1}(x^{p},z^{p},q^{p})&=\overline{\Phi}_{l-1}(x^{p},z^{p},q^{p})\overline{\Phi}_{s-l}(x^{p^{l}},z^{p^{l}},1)
\end{align}
We can multiply by $\Delta(x)$ on both sides of (\ref{split1}) and (\ref{split3}) to get:
\begin{align}\label{numlhs}
   \Phi_{s+1}(x,z,q)&=\Phi_{l}(x,z,q)\overline{\Phi}_{s-l+1}(x^{p^{l}},z^{p^{l}},1) 
\end{align}
\begin{align}\label{numrhs}
    \Phi_{s}(x,z,q)&=\Phi_{l}(x,z,q)\overline{\Phi}_{s-l}(x^{p^{l}},z^{p^{l}},1)
\end{align}
We can also multiply by $\Delta(x^{p})$ on both sides of (\ref{reqd1}) and (\ref{reqd2}) to get:
\begin{align}\label{denlhs}
    \Phi_{s}(x^{p},z^{p},q^{p})&=\Phi_{l-1}(x^{p},z^{p},q^{p})\overline{\Phi}_{s-l+1}(x^{p^{l}},z^{p^{l}},1)
\end{align}
\begin{align}\label{denrhs}
    \Phi_{s-1}(x^{p},z^{p},q^{p})&=\Phi_{l-1}(x^{p},z^{p},q^{p})\overline{\Phi}_{s-l}(x^{p^{l}},z^{p^{l}},1)
\end{align}
Dividing (\ref{numlhs}) by (\ref{denlhs}), and (\ref{numrhs}) by (\ref{denrhs}) we conclude
\begin{align}\label{Fcong}
\frac{\Phi_{s+1}(x,z,q)}{\Phi_{s}(x^{p},z^{p},q^{p})}=\frac{\Phi_{s}(x,z,q)}{\Phi_{s-1}(x^{p},z^{p},q^{p})}
\end{align}
Let $d=(d^{(i)}_{j})_{1\leq i\leq n-1,\ 1\leq j\leq \text{v}_{i}}$ with $d^{(i)}_j=j$ be a degree vector and let
\bean \label{tprimes} T'_{s}(z):=\text{coeff}_{x^{dp^{s}-1}}(\Phi_{s}(x,z)), \ \ 
T'_{s}(z,q):=\text{coeff}_{x^{dp^{s}-1}}(\Phi_{s}(x,z,q)),
\eean 
i.e., the coefficients of the monomial
\[x^{dp^{s}-1}:=\prod_{i=1}^{n-1}\prod_{j=1}^{\text{v}_{i}}x_{ij}^{jp^{s}-1}\]
The last polynomial is equal to $T_s(z,q)$ up to the sign in (\ref{tqdef}).
\\
From (\ref{numrhs}) we see that the coefficient of the monomial $x_{ij}^{jp^{s}-1}$ in $\Phi_{s}(x,z,q)$ comes from $\Phi_{l}(x,z,q)$ and $\overline{\Phi}_{s-l}(x^{p^{l}},z^{p^{l}},1)$. Namely we have
\begin{multline} \label{sumoversplits} 
T'_{s}(x,z,q)=\text{coeff}_{x^{dp^{s}-1}}\Big(\Phi_{s}(x,z,q)\Big) =  \sum\limits_{u,v,\alpha,\beta}\, 
\text{coeff}_{x^{u p^{l}-\alpha}}\Big(\Phi_{l}(x,z,q)\Big) \times\\ \text{coeff}_{x^{p^l(v p^{s-l}-\beta)}}\Big(\Phi_{s-l}(x^{p^l},1,z^{p^l})\Big)
\end{multline}
where the sum is over vectors $u$, $v$, $\alpha$, $\beta$ such that 
\[up^{l}-\alpha+p^{l}(vp^{s-l}-\beta)=dp^{s}-1\]
with $1\leq u^{(i)}_{j},v^{(i)}_{j}\leq \text{v}_{i}\ \forall 1\leq j\leq \text{v}_{i}$, and $\alpha=(\alpha^{(i)}_{j}),\ \beta=(\beta^{(i)}_{j})$ are vectors of residues modulo $p^{l}$ and $p^{s-l}$ respectively, and $1$ denotes the degree vector with all entries equal to 1.

Reducing each entry of the degree vectors on both sides modulo $p^{l}$, we find that $\alpha\equiv 1 (\text{mod }p^{l})$, so $\alpha=1$ (the degree vector with each entry equal to 1).\\
Since $u$ is a degree vector $u^{(i)}_{j}$ where $1\leq i\leq n-1$ and $1\leq j\leq \text{v}_{i}$, we know by (\ref{degbound}) that the degree of each $x_{ij}$ in $\Phi_{l}$ is less than $(\text{v}_{i}+1)p^{l}-1$. We also know that $\Phi_{l}$ is skew symmetric (under the action of $S_{\text{v}}$). Hence for a fixed value of $i$, the coefficient of $x^{up^{l}-1}$ in $\Phi_{l}(x,z,q)$ is nonzero only if all the $u^{(i)}_{j}$ are distinct, with $1\leq u^{(i)}_{j}\leq \text{v}_{i}$. Hence for each fixed $i$, the $u^{(i)}_{j}$ are a permutation of $1,\dots,\text{v}_{i}$. Thus, $u=\sigma(d)$ for some $\sigma \in S_{\text{v}}$, and
\[\text{coeff}_{x^{u p^{l}-1}}(\Phi_{l}(x,z,q))=\text{coeff}_{x^{\sigma(d)p^{l}-1}}(\Phi_{l}(x,z,q))=\epsilon(\sigma)\text{coeff}_{x^{dp^{l}-1}}(\Phi_{l}(x,z,q))=\epsilon(\sigma)T'_{l}(z,q)\]
where $\epsilon(\sigma)$ is the sign of the permutation $\sigma$. 
We then need to multiply this coefficient with $\text{coeff}_{x^{dp^{s}-1-up^{l}-1}}(\ovl{\Phi}_{s-l}(x^{p^{l}},z^{p^{l}},1))$ and sum over all degree vectors $u,v,\beta$ with the properties mentioned previously to obtain $\Phi_{s}(x,z,q)$.\\
Thus the first multiple in the sum (\ref{sumoversplits}) factors out and we obtain:
$$
T'_{s}(z,q)=T'_{l}(z,q) F_{1}(z^{p^{l}})
$$
where $F_{1}(z^{p^{l}})$ is a certain polynomial in $z$ independent of $q$. Similarly, repeating this argument for all terms in (\ref{split1})-(\ref{split4}) we obtain:
\bean
\begin{array}{l}
T'_{s+1}(z,q)=T'_{l}(z,q)F_{1}(z^{p^{l}})\\
T'_{s}(z^{p},q^{p})=T'_{l-1}(z^{p},q^{p})G_{1}(z^{p^{l}})\\
T'_{s}(z,q)=T'_{l}(z,q)F_{2}(z^{p^{l}})\\
T'_{s-1}(z^{p},q^{p})=T'_{l-1}(z^{p},q^{p})G_{2}(z^{p^{l}})
\end{array}
\eean
for some polynomials $G_{i}(z^{p^l})$ and $F_{i}(z^{p^l})$ independent of $q$. \\
Suppose $x$ appears with exponent $\alpha p^{l}-1$ in $\Phi_{l}(x,z,q)$, where $\alpha=(\alpha^{(i)}_{j}$) is a degree vector so that for each $i$ the $\alpha^{(i)}_{j}$ are a permutation of $1,\dots,\text{v}_{i}$. Let $\sigma\in S_{\text{v}}$ be the permutation so that $\sigma\cdot((t^{(i)}_{j}))=\alpha$, where $1\leq i\leq n-1,\ 1\leq j\leq v_{i}$, and $t^{(i)}_{j}=j$. The coefficient of this monomial is $\epsilon(\sigma)T'_{l}(z,q)$.

If such a monomial contributes to the coefficient of $x^{dp^{s+1}-1}$ in $\Phi_{s+1}(x,z,q)$ (as in (\ref{sumoversplits})), then we know that $x$ appears with exponent $p^{l}(mp^{s+1-l}+\beta)$ in $\Phi_{s+1-l}(x^{p^{l}},z^{p^{l}},1)$ so that\\
\[\alpha p^{l}-1+p^{l}(mp^{s-l+1}+\beta)=dp^{s+1}-1\]
The coefficient of this monomial also appears in $F_{1}(z^{p^{l}})$ by the definition of $F_{1}$. We will now show that this term appears in $G_{1}$ as well.

If we add the degree vector whose entries are all 1 to both sides, divide by $p$, subtract the degree vector 1, then multiply by $p$ again, we now have
\begin{align}\label{finaldegarg}p(\alpha p^{l-1}-1)+p^{l}(mp^{s-l+1}+\beta)=p(dp^{s}-1)\end{align}
Let us now consider (\ref{denlhs}). The coefficient of the monomial $(x^{p})^{\alpha p^{l-1}-1}$ in $\Phi_{l-1}(x^{p},z^{p},q^{p})$ is $\epsilon(\sigma)T'_{l-1}(z^{p},q^{p})$ for the same reason as in our previous computation of $\text{coeff}_{x^{\alpha p^{l}-1}}\Phi_{l}(x,z,q)$. We also know that 
\begin{multline*}\text{coeff}_{x^{p(dp^{s}-1)}}\Big(\Phi_{s}(x^{p},z^{p},q^{p})\Big) =  \sum\limits_{u',v',\alpha',\beta'}\, 
\text{coeff}_{x^{p(u'p^{l-1}-\alpha')}}\Big(\Phi_{l-1}(x^{p},z^{p},q^{p})\Big)\times\\ \text{coeff}_{x^{p(p^{l-1}(v' p^{s-l+1}-\beta'))}}\Big(\Phi_{s-l+1}(x^{p^l},1,z^{p^l})\Big)\end{multline*}This is analogous to (\ref{sumoversplits}). The LHS of this equation is $T'_{s}(z^{p},q^{p})$. The summation on the right is over degree vectors $u', v'$ and vectors of residues $\alpha',\ \beta'$ modulo $p^{l-1}$ and $p^{s-l+1}$ respectively such that \[p(u'p^{l-1}-\alpha')+p^{l}(v'p^{s-l+1}-\beta')=p(dp^{s}-1)\] 
We can divide by both sides by $p$, reduce each entry of the degree vectors on both sides modulo $p^{l-1}$ and conclude that each entry of $\alpha'$ is 1. Hence  \[u'p^{l-1}-1+p^{l-1}(v'p^{s-l+1}-\beta')=dp^{s}-1\] Adding the vector 1 to both sides, multiplying by $p$, and subtracting the vector 1 from both sides we have \[u'p^{l}-1+p^{l}(v'p^{s-l+1}-\beta')=dp^{s+1}-1\] as before. Recall that we previously chose a degree vector $\alpha$ so that the coefficient of $x^{\alpha p^{l}-1}$ in $\Phi_{l}(x,z,q)$ contributed to $T'_{s+1}(z,q)$. The coefficient of $x^{p(\alpha p^{l-1}-1)}$ in $\Phi_{s}(x^{p},z^{p},q^{p})$ is $\epsilon(\sigma)T'_{l-1}(z^{p},q^{p})$ as before. The coefficient of the monomial $x^{\xi}$ of $\ovl{\Phi}_{s-l+1}(x^{p^{l}},z^{p^{l}},1)$ so that \[\xi+p(\alpha p^{l-1}-1)=p(dp^{s}-1)\] (note that this appears in $G_{1}$ by definition) is exactly the coefficient of $x^{p^{l}(mp^{s-l+1}+\beta)}$, which is exactly the term that appears in $F_{1}$. Hence every term in $F_{1}$ occurs in $G_{1}$, and we can repeat the argument to show that every term in $G_{1}$ appears in $F_{1}$ as well, which implies that $F_{1}=G_{1}$. We can repeat the argument to show that $F_{2}=G_{2}$ as well. So we have
\[\frac{T'_{s+1}(z,q)}{T'_{s}(z^{p},q^{p})}=\frac{T'_{s}(z,q)}{T'_{s-1}(z^{p},q^{p})}\] 
when $q$ is an arbitrary primitive $p^{l}$-th root of unity, where $1\leq l<s$.  Finally,  by (\ref{tqdef}) $T'_s(z,q)=(-1)^{(p^{s}-1)N\omega} T_s(z,q)$, so 
\begin{align*}
    \frac{T_{s+1}(z,q)}{T_{s}(z^{p},q^p)}=(-1)^{(p^{s+1}-p^{s})N\omega}\frac{T'_{s+1}(z,q)}{T'_{s}(z^{p},q^p)}\\
    \frac{T_{s}(z,q)}{T_{s-1}(z^{p},q^p)}=(-1)^{(p^{s}-p^{s-1})N\omega}\frac{T'_{s}(z,q)}{T'_{s-1}(z^{p},q^p)}
\end{align*}
And we have $(-1)^{(p^{s+1}-p^{s})N\omega}=(-1)^{(p-1)N\omega}=(-1)^{(p^{s}-p^{s-1})N\omega}$ (since $p$ is an odd prime). Hence 
\[\frac{T_{s+1}(z,q)}{T_{s}(z^{p},q^{p})}=\frac{T_{s}(z,q)}{T_{s-1}(z^{p},q^{p})}\] 
which finishes the proof. 
\end{proof}

\begin{cor}
The polynomials $T_s(z)$ defined by (\ref{tpols}) satisfy the Dwork type congruences 
\bean \label{svcorr}
\dfrac{T_{s+1}(z)}{T_{s}(z^p)} \equiv  \dfrac{T_{s}(z)}{T_{s-1}(z^p)}  \pmod{p^s} 
\eean 
\end{cor}
\begin{proof}

Specializing (\ref{mainform}) to $q=1$, and noting that $[p^s]_q$ specializes in this limit to $p^s$ we obtain 
\[(-1)^{\theta_{s+1}-\theta_{s}}\frac{T_{s+1}(z)}{T_{s}(z^{p})}\equiv(-1)^{\theta_{s}-\theta_{s-1}}\frac{T_{s}(z)}{T_{s-1}(z^{p})}(\text{ mod }p^{s})\]
where the sign appears as discussed at the end of Section~\ref{signref}.  Thus
\[(-1)^{\theta_{s+1}-2\theta_{s}+\theta_{s-1}}\frac{T_{s+1}(z)}{T_{s}(z^{p})}\equiv\frac{T_{s+1}(z)}{T_{s}(z^{p})}\equiv\frac{T_{s}(z)}{T_{s-1}(z^{p})}\text{(mod }p^{s})\]
This follows since $2\theta_{s}$ is even, and $(-1)^{\theta_{s+1}+\theta_{s-1}}=(-1)^{\theta_{s+1}-\theta_{s-1}}=1$, since 
\[\theta_{s+1}-\theta_{s-1}=\left(\sum_{m=1}^{n-1}\frac{v_{m}(v_{m}+1)}{2}\right)\frac{p^{s-1}(p^{2}-1)(1-2\omega)}{2}\in 2\mathbb{Z}\]
Hence we arrive at (\ref{svcorr}). 
\end{proof}

\end{document}